\documentclass[11pt]{amsart}
\usepackage{cite,hyperref,subcaption}
\usepackage[width=5.5in,height=8.5in,centering]{geometry}
\usepackage{amssymb,amsmath,amsthm,mathtools,extpfeil}
\usepackage{tikz}
\usepackage{enumerate}

\usetikzlibrary{snakes}

\newtheorem{thm}[equation]{Theorem}
\newtheorem*{thm*}{Theorem}
\newtheorem{thmA}{Theorem}

\newtheorem{lem}[equation]{Lemma}
\newtheorem{prop}[equation]{Proposition}
\newtheorem{cor}[equation]{Corollary}

\theoremstyle{definition}

\newtheorem{rmk}[equation]{Remark}

\numberwithin{equation}{section}

\DeclareMathOperator{\Gr}{Gr}
\DeclareMathOperator{\Lip}{Lip}
\DeclareMathOperator{\Var}{Var}
\DeclareMathOperator{\AKT}{AKT}
\DeclareMathOperator{\mass}{mass}
\DeclareMathOperator{\vol}{vol}
\DeclareMathOperator{\sgn}{sign}

\DeclareMathOperator{\conc}{conc}

\newcommand{\ph}{\varphi}
\newcommand{\epsi}{\varepsilon}

\newcommand{\elbow}{\mathbin{\llcorner}}

\begin{document}
\title{Filling random cycles}
\author{Fedor Manin}
\address{Department of Mathematics, UCSB, Santa Barbara, California, USA}
\email{manin@math.ucsb.edu}
\begin{abstract}
  We compute the asymptotic behavior of the average-case filling volume for
  certain models of random Lipschitz cycles in the unit cube and sphere.  For
  example, we estimate the minimal area of a Seifert surface for a model of
  random knots first studied by Millett.  This is a generalization of the
  classical Ajtai--Koml\'os--Tusn\'ady optimal matching theorem from
  combinatorial probability.  The author hopes for applications to the topology
  of random links, random maps between spheres, and other models of random
  geometric objects.
\end{abstract}
\maketitle

\section{Introduction}

\subsection{Main results}
This paper introduces a new kind of average-case isoperimetric inequality.  Given
a $k$-cycle $Z$ on $([0,1]^n,\partial[0,1]^n)$, in any of a number of geometric
measure theory senses, its \emph{filling volume} $FV(Z)$ is the minimal mass of a
chain whose boundary is $Z$.  The well-known Federer--Fleming isoperimetric
inequality \cite{FF} states that for all $k$-cycles $Z$,
\[FV(Z) \leq C_{n,k}\mass(Z)^{\frac{k+1}{k}}\quad\text{and}\quad
FV(Z) \leq C_{n,k}\mass(Z).\footnote{The first inequality dominates when
  $\mass(Z)<<1$, the second when $\mass(Z)>>1$.}\]
However, one might expect that \emph{most} cycles of given mass are much easier
to fill.

Unfortunately, as explained to the author by Robert Young, a geometrically
meaningful probability measure on the space of all cycles of mass $\leq N$ may be
too much to ask for.  The issue is one of picking a scale: say we are trying to
build a random 1-Lipschitz curve in a finite-dimensional space.  If the curve is
to fluctuate randomly at scale $\epsi$, then over time $1$ it will only travel a
distance on the order of $\sqrt{\epsi}$.  Thus there is no way of ensuring
random behavior in a scale-free way.  This idea of decomposing a finite-mass
cycle into pieces at different scales can be made precise using the notion of a
\emph{corona decomposition}, as in \cite{Jones} (in dimension 1) and \cite{Young}
(in higher dimensions).

On the other hand, there are a number of ways of, and a number of motivations
for, building random ``space-filling'' cycles of mass $O(N)$ which look
essentially trivial on balls of radius $N^{-1/n}$.  Our main theorems characterize
three models of this form which exhibit similar isoperimetric behavior, and we
hypothesize that this behavior should be generic for models which are random at
many scales, including the largest---an idea which may have a precise
Fourier-analytic formulation.

This isoperimetric behavior is described in codimension $d$ by the function
\[\AKT_d(N)=\begin{cases}
  \sqrt{N} & \text{if }d=1 \\
  \sqrt{N\log N} & \text{if }d=2 \\
  N^{(d-1)/d} & \text{if }d \geq 3. \\
\end{cases}\]
\begin{thmA} \label{sphere}
  Let $Z$ be a $k$-cycle on $S^n$ obtained by sampling $N$ oriented great
  $k$-spheres independently from the uniform distribution on the oriented
  Grassmannian $\widetilde\Gr_{k+1}(\mathbb{R}^{n+1})$.  Then there are constants
  $C>c>0$ depending on $n$ and $k$ such that
  \begin{equation} \label{AKTstats:E}
    c\AKT_{n-k}(N)<\mathbb{E}(FV(Z))<C\AKT_{n-k}(N).
  \end{equation}
  Moreover, $FV(Z)$ is concentrated around its mean: there are constants
  $C_1,C_2>0$ depending on $n$ and $k$ such that
  \begin{equation} \label{AKTstats:V}
    \text{for every $r>0$, }
    \mathbb{P}[\lvert FV(Z)-\mathbb{E}(FV(Z))\rvert \geq r]
    \leq C_1\exp(-C_2\sqrt{N}r)
  \end{equation}
\end{thmA}
From \eqref{AKTstats:V} we see that the spread of the distribution around the
mean is at most on the order of $\sqrt{N}$; in codimensions $d=n-k \geq 2$, this
is small compared to the mean:
\[\text{for every $\epsi>0$, }
\frac{\lvert FV(Z)-\mathbb{E}(FV(Z))\rvert}{\mathbb{E}(FV(Z))}
\leq \epsi \text{ with high probability as }N \to \infty.\]
If a model of random codimension-$d$ cycles of mass $O(N)$ satisfies
\eqref{AKTstats:E} and \eqref{AKTstats:V}, we say it
\emph{exhibits AKT statistics}, in honor of Ajtai, Koml\'os, and Tusn\'ady, who
discovered this phenomenon in the case of zero-cycles.

By rescaling the picture, we can make this result more interpretable.  Let
$R=N^{1/(n-k)}$.  Then the corresponding process in the $n$-sphere of radius $R$
generates a cycle of mass $\Theta(R^n)$ which is evenly spread throughout the
sphere, so that a 1-ball intersects one of the great $k$-spheres in $Z$ on
average.  With that rescaling, the mass of an optimal filling becomes
\[\left\{\begin{array}{l l}
\Theta(R^n\sqrt{R}) & \text{if }k=n-1 \\
\Theta(R^n\sqrt{\log R}) & \text{if }k=n-2 \\
\Theta(R^n) & \text{if }k \leq n-3. \\
\end{array}\right.\]
Informally speaking, to meet its match, the average point in $Z$ has to travel a
distance $\Theta(\sqrt{R})$ (in codimension 1), $\Theta(\sqrt{\log R})$ (in
codimension 2), or $\Theta(1)$ (otherwise) times the distance to its closest
neighbor.

We also prove similar results for the cube:
\begin{thmA} \label{cube}
  Let $Z$ be a relative $k$-cycle on $([0,1]^n,\partial[0,1]^n)$ obtained by
  sampling $N$ planes independently from the uniform distribution on the space
  $Y$ of oriented $k$-planes which intersect $[0,1]^n$ nontrivially.  Then $Z$
  exhibits AKT statistics.
\end{thmA}
There may be reasonable disagreement as to which distribution is the uniform one
in this context; to prove \eqref{AKTstats:E} it suffices to require that it be
uniform on each subset of $Y$ (isometric to two copies of a polytope) consisting
of parallel planes, but to prove \eqref{AKTstats:V} we also need to assume that
it behaves reasonably with respect to the manifold structure on $Y$ (for example,
is a positive density or has finite support).

In fact, the only thing used here about $Z$ is that almost all of its ``slices''
along coordinate $k$-planes consist of $O(N)$ independent uniformly distributed
points.  This means that there are a number of other possible models that can be
fit into this framework.  However, the following requires a separate proof:
\begin{thmA} \label{knot}
  Let $\{M_N\}$ be a sequence of $k$-dimensional oriented pseudomanifolds with
  $N$ vertices and at most $L$ simplices incident to any given simplex.  Let $Z$
  be a $k$-cycle on $[0,1]^n$ obtained by sending each vertex of $M_N$ to a
  uniformly random point in $[0,1]^n$ and extending linearly.  Then $Z$ exhibits
  AKT statistics.
\end{thmA}
In the context of this theorem, the constants in \eqref{AKTstats:E} depend on
$n$, $k$, and $L$, but not on the shapes of the pseudomanifolds (which can
therefore also be randomized).  The case $k=1$, $n=3$ describes the ``random
jump'' model of random knots and links introduced by Millett \cite{Millett}.
Moreover, by a theorem of Hardt and Simon \cite{HS}, the optimal filling of such
a knot or link (after a slight rounding of corners) is a $C^1$ embedded surface.
In particular:
\begin{cor}
  For some $C>c>0$, the minimal Seifert surface of a knot produced using $N$
  random jumps has area between $c\sqrt{N\log N}$ and $C\sqrt{N\log N}$ with
  high probability.
\end{cor}

\subsection{Motivation}
The methods we use to prove Theorems \ref{sphere}, \ref{cube}, and \ref{knot} can
be easily extended to other i.i.d.\ samples of simple shapes on various spaces.
However, the investigation is mainly motivated by the desire to analyze
topological invariants of random geometric objects such as links and maps.
Models of such objects tend to produce random cycles which are similarly trivial
at small scales, but are more difficult to sample because they cannot be easily
written in terms of i.i.d.\ parameters.

\subsubsection*{Random knots and links}
There have been a number of proposed models of random knots and links; see
\cite{EZ} for a detailed survey.  Several of these models are ``spatial'' in the
sense that they produce random knotted curves in space, and one supposes that
these may exhibit AKT statistics for filling area.  As mentioned above, we show
this for Millett's random jump model, but it may also be true for random
polygonal walks with shorter segments as well as random grid walks, perhaps with
some restrictions on segment length.

Given two random curves in a certain model, one may want to understand the
distribution of their linking number.  Since this will usually be zero on
average, the first interesting question is about the second moment.  Linking
number can be computed as the intersection number of one curve with a filling of
the other, thus one may expect that two random curves of length $N$ which exhibit
AKT statistics have expected squared linking number $\sim N\log N$ or
$\sim N\sqrt{\log N}$.

However, this is not the case for the Millett model: the second moment of the
linking number between two random jump curves of length $N$ is $\sim N$
\cite{ABDKS,FlKo}.  Similarly, one may take the setup of Theorem \ref{sphere} for
$k=1$ and $n=3$ as a model of a random link and try to understand the
\emph{total linking number}, that is, the sum of the signed linking numbers of
all pairs of circles.  This is then the intersection number of the chain with its
own filling.  Here it is easy to see (as pointed out by Matthew Kahle) that the
second moment of the distribution is once again $\sim N$.

In both cases, this seeming incongruity perhaps boils down once again to the
issue of multiple scales: random jump curves and great circles only ``see'' the
largest scales, but the lower bound on filling volume in codimension 2 comes from
looking on many different scales at once.  One may perhaps get a different answer
most easily by analyzing the linking number of an asymmetric model: a random jump
curve and a random walk of total length $N$ made of smaller segments.

In \cite{Tanaka,Marko}, the second moment is computed for the linking number of
two random walks; normalizing so that these walks have length $N$ and expected
diameter $1$, this second moment again becomes $\sim N$.  In this model, however,
randomness happens at scale $\sim 1/N$, so it is not expected to exhibit AKT
statistics.

\subsubsection*{Random maps}
Another way of producing a random (framed) link is as the preimage of a generic
point under a random map $f:S^3 \to S^2$.  In fact, the self-linking number of
this link is the Hopf invariant of the map, which is itself a natural subject for
investigation since it is a complete topological invariant of such maps.

One natural model of $L$-Lipschitz random maps is a uniformly random
\emph{simplicial} map from a triangulation of $S^3$ at scale $\sim L$ to a
tetrahedron.  The maximal self-linking number of such a map is $\Theta(L^4)$,
cf.\ \cite{GroHED}; on the other hand, the heuristics above would suggest that
the second moment of the linking number of the random model is between $L^3$ and
$L^3\log L$.

These ideas may have applications in topological and geometric data analysis, see
\cite{FKW}.

\subsection{Methods}
The $k=0$ cases of Theorems \ref{sphere} and \ref{cube} are, up to minor
adjustments, a classical theorem in combinatorial probability:
\begin{thm}[Ajtai, Koml\'os, and Tusn\'ady \cite{AKT}] \label{thm:AKT}
  Let $\{X_1,\ldots,X_N\}$ and $\{Y_1,\ldots,Y_N\}$ be two sets of independent,
  uniformly distributed random points in $[0,1]^d$, and let $L$ be the
  \emph{transportation cost} between $\{X_i\}$ and $\{Y_i\}$, that is, the total
  length of an optimal matching.  Then there are constants $0<c_d<C_d$
  such that with high probability,
  \[c_d\AKT_d(N)<L<C_d\AKT_d(N).\]
\end{thm}
Since the original geometric proof in \cite{AKT} of the most subtle case $d=2$,
this and related results have been proved many other times, often by applying
Fourier analysis; see \cite{BobL} for further references and \cite{TalBk} for a
detailed treatment of certain analytic approaches.  Another beautiful geometric
proof of the upper bound on the sphere is due to Holden, Peres, and Zhai
\cite{HPZ}.

The proofs of Theorems \ref{sphere} and \ref{cube} in general are obtained by
applying the $k=0$ results to $(n-k)$-dimensional slices of the cube and sphere.
This is the reason that the results depend only on the codimension, and for the
critical nature of codimension 2.  The lower bound in \eqref{AKTstats:E} is
obtained directly by integrating the lower bounds on these slices.  The upper
bound is obtained via a dual result on differential forms; this kind of technique
was already used in \cite{AKT} for the proof of the lower bound for the square.
Finally, \eqref{AKTstats:V} is proved using the notion of concentration of
measure due originally to Gromov and Milman \cite{GroMi}; see \cite{Ledoux} for
an extensive modern treatment.

Theorem \ref{knot} is proved similarly, except that slices no longer consist of
i.i.d.\ points.  Even this small amount of dependence complicates the argument
considerably.  We use ad hoc combinatorial arguments to overcome this, but one
might hope to generalize, for example by applying a variant of Stein's method, to
a version of Theorem \ref{thm:AKT} in the presence of dependence (one approach,
which only gives upper bounds, is discussed in \cite[\S5]{BobL}).

\subsection*{Structure of the paper}
Section \ref{S2} introduces necessary ideas and results from geometric measure
theory, and Section \ref{S3} discusses the classical AKT theorem.  In Sections
\ref{S:upper} and \ref{S:lower}, the upper and lower bounds in Theorems
\ref{sphere} and \ref{cube} are proved using tools that may generalize to other
models of random cycles.  In Section \ref{S:knot}, we discuss the extra ideas
needed to prove Theorem \ref{knot}.  Finally, Section \ref{S:concentration}
discusses the concentration of the distributions in these theorems around their
mean.

\subsection*{Acknowledgements}
I would like to thank Matthew Kahle and Robert Young for a large number of
helpful discussions over a span of three years.  Yevgeny Liokumovich provided a
crucial reference; Shmuel Weinberger asked a question which inspired Theorem
\ref{knot} and gave other helpful comments.  Finally, I would like to thank both
Robert Young and the anonymous referee for calling out a great deal of sloppy and
lazy writing in the initial draft.  I was partially supported by NSF individual
grant DMS-2001042.

\section{Definitions and preliminaries} \label{S2}

\subsection{Cycles and currents}
There are a number of useful ways to define chains and cycles from the point of
view of topology and geometric measure theory.  Algebraic topology typically uses
singular $k$-chains: formal linear combinations of continuous maps from the
$k$-simplex to a topological space $X$ (``singular simplices'').  We will usually
restrict our attention to Lipschitz simplices (that is, requiring the maps to be
Lipschitz) on a Riemannian manifold $M$.  By Rademacher's theorem, a Lipschitz
simplex $\sigma:\Delta^k \to M$ is differentiable almost everywhere and so has a
well-defined volume or \emph{mass},
\[\mass(\sigma)=\int_{\Delta^k} \sigma^*d\vol_M.\]
We can then extend by linearity to define the mass of a Lipschitz chain.

%% \begin{lem} \label{Sard}
%%   A smooth $k$-chain $T$ in $\mathbb{R}^n$ defines a continuous family (in the
%%   flat metric) of smooth $(k-m)$-chains in $\mathbb{R}^{n-m}$ parametrized by an
%%   open dense subset of $\mathbb{R}^m$.
%% \end{lem}
%% \begin{proof}
%%   Let $T=\sum_{i=1}^r a_i\sigma_i$, where the $\sigma_i$ are smooth maps
%%   $\Delta^k \to \mathbb{R}^n$.  Write $\pi$ for the projection
%%   $\mathbb{R}^n \to \mathbb{R}^m$.

%%   Since $\Delta^k$ is compact, the set of critical values of $\pi \circ \sigma_i$
%%   is also compact.  By Sard's theorem it also has measure zero, which implies
%%   that it is nowhere dense.  Thus there is an open dense subset
%%   $U \subset \mathbb{R}^m$ consisting of points which are regular values for
%%   every $\pi \circ \sigma_i$.  The preimage of a regular value under $\sigma_i$
%%   is a smoothly triangulable $(k-m)$-submanifold of $\Delta^k$.  This lets us
%%   define a $(k-m)$-chain $T_x$ for every $x \in U$.

%%   Now consider the mapping $x \mapsto T_x$ near some point $x_0$.  Since $x_0$ is
%%   a regular value of each $\sigma_i$, we can choose a closed ball $B=B_\epsi(x_0)$
%%   such that for each $i$, $\sigma_i^{-1}(B)$ is diffeomorphic to
%%   $B \times \sigma_i^{-1}(x_0)$.  Then the maximal bilipschitz constant of these
%%   diffeomorphisms together with the area of $\sigma_i^{-1}(x_0)$ controls the
%%   Lipschitz constant of $x \mapsto T_x$ restricted to $B$ with respect to the
%%   flat metric.
%% \end{proof}

A more general notion of chain is that of a normal current.  A $k$-dimensional
\emph{current} on a manifold $M$ is simply a functional on (smooth) differential
forms, which we think of as integration over the current.  For example:
\begin{itemize}
\item Every Lipschitz chain $T$ defines a current via
  $\omega \mapsto \int_T \omega$.
\item Every compactly supported $(n-k)$-form $\alpha \in \Omega^{n-k}(M)$ defines
  a current via $\omega \mapsto \int_M \alpha \wedge \omega$.
\end{itemize}
We will write the value of $T$ on $\omega$ either as $T(\omega)$ or as
$\int_T \omega$, since currents should be thought of as generalized domains of
integration.  The boundary operator is defined via Stokes' theorem: for a current
$T$,
\[\partial T(\omega)=T(d\omega).\]
The \emph{mass} of a $k$-current $T$ on $M$, which agrees with the same notion on
Lipschitz chains, is defined to be
\[\mass(T)=\inf \{T(\omega) : \omega \in \Omega^k(M)\text{ and }
\lVert\omega\rVert_\infty=1\}.\]
Here $\lVert\omega\rVert_\infty$ is the supremum of the value of $\omega$ over all
frames of unit vectors.  For a general current, the mass of course need not be
finite.  A current $T$ is \emph{normal} if $T$ and $\partial T$ both have finite
mass; in particular any cycle (current with empty boundary) of finite mass is
normal.

\subsection{Fillings and duality}
Now, if $S$ is a normal current such that $\partial S=T$, we call it a
\emph{filling} of $T$.  The \emph{filling volume} of $T$ is
\[FV(T)=\inf \{\mass(S) \mid \partial S=T\},\]
which is always finite by the work of Federer and Fleming.  The following is an
instance of the Hahn--Banach theorem:
\begin{prop} \label{duality}
  Let $M$ be a manifold.  Then a normal $k$-current $T$ in $M$ with
  $\partial T=0$ has a filling of mass $c$ if and only if for every
  $\omega \in \Omega^k(M)$ with $\lVert d\omega\rVert_\infty \leq 1$,
  $\int_T \omega \leq c$.

  More generally, for any closed set $A \subset M$, write $\Omega^k(M,A)$ for the
  vector space of forms whose restriction to $A$ is zero.  Let $T$ be a normal
  $k$-current with $\partial T$ supported on $A$, that is, such that
  $\int_{\partial T} \alpha=0$ for any $(k-1)$-form $\alpha \in \Omega^{k-1}(M,A)$.
  Then $T$ has a filling relative to $A$ (that is, a $(k+1)$-current $S$ such
  that $\partial S-T$ is supported on $A$) of mass $c$ if and only if for every
  $\omega \in \Omega^k(M,A)$ with $\lVert d\omega\rVert_\infty \leq 1$,
  $\int_T \omega \leq c$.
\end{prop}
In other words, the filling volume of a cycle $T$ can be redefined as
\[FV(T)=\sup \bigl\{\textstyle{\int_T \omega} \mid \omega \in \Omega^k(M)
\text{ such that }\lVert d\omega \rVert_\infty \leq 1\bigr\}\]
in both the absolute and the relative case.  Our proofs of the upper bounds in
Theorems \ref{sphere} and \ref{cube} will be based on this proposition rather
than constructing fillings directly.

Of course, knowing that a nice Lipschitz cycle has a filling which is a normal
current is not very satisfying---after all, normal currents can still be very
strange.  Luckily, given a normal current filling, we can upgrade it to a
Lipschitz chain (at the cost of multiplying the mass by a constant) using the
following classical theorem:
\begin{thm}[{Federer--Fleming deformation theorem \cite[Thm.~5.5]{FF}}]
  There is a constant $\rho(k,n)=2n^{2k+2}$ such that the following holds.  Let
  $T$ be a normal current in $\mathbf{N}_k(\mathbb{R}^n)$.  Then for every
  $\epsi>0$ we can write $T=P+Q+\partial S$, where
  \begin{enumerate}
  \item $\mass(P) \leq \rho(k,n)\mass(T)$.
  \item $\mass(Q) \leq \epsi\rho(k,n)\mass(\partial T)$.
  \item $\mass(S) \leq \epsi\rho(k,n)\mass(T)$.
  \item $P$ is a polyhedral cycle which can be expressed as an
    $\mathbb{R}$-linear combination of $k$-cells in the cubical unit
    lattice in $\mathbb{R}^n$.
  \item If $T$ is a Lipschitz chain, then so are $Q$ and $S$.
  \item If $\partial T$ is a Lipschitz chain, then so is $Q$.
  \end{enumerate}
\end{thm}
If $T$ is a normal current filling a Lipschitz chain $\partial T$, then $P+Q$ is
a Lipschitz chain filling $T$ whose mass is only greater by a multiplicative
constant $\rho(k,n)$.

It is not hard to upgrade the deformation theorem to manifolds, although the
resulting constants will depend on the manifold and its metric; see for example
\cite[Theorem 10.3.3]{ECHLP}.

\subsection{Slicing}
An important property of normal currents, introduced in \cite[\S3]{FF}, is the
ability to take ``slices'' by hyperplanes to produce currents in lower
dimensions.  We follow the exposition of F.~Morgan \cite[4.11]{Morgan}, who
follows Federer \cite[\S4.2.1]{FedBk}.

Let $u:M \to \mathbb{R}$ be a Lipschitz function on a manifold $M$.  Given a
$k$-current $T$ on $M$ and a differential $r$-form $\omega$, define the
$(k-r)$-current $T \elbow \omega$ by
\[T \elbow \omega(\eta)=T(\omega \wedge \eta).\]
In particular, this makes sense when $\omega$ is a measurable function, for
example the characteristic function $\chi_A$ of a set $A$.  In that case we can
write $T \elbow A=T \elbow \chi_A$ for the restriction of $T$ to $A$.

Given a Lipschitz function $u:M \to \mathbb{R}$, the \emph{slice} of $T$ at
$u(x)=r$ is defined by
\begin{equation} \label{def-of-slice}
  T \cap \{u(x)=r\}=(\partial T)\elbow\{u(x)>r\}-\partial(T\elbow\{u(x)>r\}).
\end{equation}
If $T$ is a normal current, then $T \cap \{u(x)=r\}$ is a normal current for
almost all $r$.  In addition, we have the following standard properties:
\begin{enumerate}
\item If $T$ is defined by integration over a (rectifiable) set $X \subset M$,
  then $T \cap \{u(x)=r\}$ is defined by integration over $M \cap \{u(x)=r\}$.
\item $\partial T \cap \{u(x)=r\}=-\partial(T \cap \{u(x)=r\})$.
\item $\mass(T) \geq \frac{1}{\Lip u}\int_{-\infty}^\infty
  \mass(T \cap \{u(x)=r\})dr.$
\end{enumerate}
A quick calculation from the definitions yields an additional property:
\begin{prop}
  If $\omega$ is a $(k-1)$-form, then
  \[T(du \wedge \omega)=-\int_{-\infty}^\infty (T \cap \{u(x)=r\})(\omega)dr.\]
\end{prop}
\begin{proof}
  Start with the equalities
  \begin{align*}
    \partial(T \elbow u)(\omega)
    &= \int_{-\infty}^\infty \partial(T \elbow \{u(x)>r\})(\omega)dr \\
    (\partial T \elbow u)(\omega)
    &= \int_{-\infty}^\infty ((\partial T) \elbow \{u(x)>r\})(\omega)dr.
  \end{align*}
  Subtract one from the other; then applying Stokes' theorem and the Leibniz rule
  on the left and \eqref{def-of-slice} on the right, we get the desired identity.
\end{proof}

In particular, by inductively slicing in different directions, we get the
following:
\begin{prop} \label{slice-coarea}
  Let $1 \leq k \leq m \leq n$, and let $T$ be an $m$-dimensional current on
  $[0,1]^n$.  Given $\vec x=(x_1,\ldots,x_k) \in \mathbb{R}^k$, let $P_{\vec x}$ be
  the plane $\{\vec x\} \times \mathbb{R}^{n-k}$.  Then there are
  $(m-k)$-dimensional currents $T \cap P_{\vec x}$, normal for almost all
  $\vec x$, such that
  \begin{gather*}
    \partial(T \cap P_{\vec x})=\partial T \cap P_{\vec x} \\
    \mass(T) \geq \int_{[0,1]^k} \mass(T \cap P_{\vec x})d\vec x.
  \end{gather*}
  In addition, given an $m$-index $I \subset \{1,\ldots,n\}$ and a function
  $f:[0,1]^n \to \mathbb{R}$,
  \[\int_T fdx_I=(-1)^m\int_{[0,1]^{I^c}} \Bigl(\int_{T \cap P_{\vec x}} f\Bigr)d\vec x.\]
\end{prop}

\section{A variation on the Ajtai--Koml\'os--Tusn\'ady theorem} \label{S3}
The results of this paper are a generalization of Theorem \ref{thm:AKT}.
Properly, the theorem of Ajtai, Koml\'os, and Tusn\'ady \cite{AKT} is in the case
$n=2$; their paper also asserts the case $n \geq 3$, which is easy and later
proved and extended in several directions by Talagrand \cite{TalHiD,TalHiD2}.
The $n=1$ case is elementary, and the proof along with a vast array of
strengthenings and generalizations can be found in \cite{BobLMem} by Bobkov and
Ledoux.  Here we need a slight variation.
\begin{thm} \label{AKT:boundary}
  Generate a cycle $Z$ of mass $N$ in $C_0([0,1]^n,\partial[0,1]^n)$ by selecting
  $N$ independent, uniformly distributed points in $[0,1]^n \times \{+1,-1\}$.
  Then there are constants $0<c_n<C_n$ such that
  \[c_n\AKT_n(N) \leq \mathbb{E}(FV(Z)) \leq C_n\AKT_n(N).\]
\end{thm}
\begin{rmk} \label{AKT:diffeo}
  Suppose that $D$ is a Riemannian ball diffeomorphic to $[0,1]^n$ and has a
  volume form.  Then by the main theorem of \cite{BMPR} (extending results of
  Moser \cite{Moser} and Banyaga \cite{Banyaga}), there is a diffeomorphism
  between the two which multiplies the volume form by a constant.  Therefore
  Theorem \ref{AKT:boundary} also holds with respect to Lebesgue measure on $D$,
  with constants $0<c_D<C_D$ depending on the ratio of the volumes and the
  bilipschitz constant of this diffeomorphism.

  Moreover, given a smooth family of Riemannian balls, \cite{BMPR} indicates that
  there is a smooth family of such diffeomorphisms.  Therefore, if the family is
  compact, one can find uniform constants for the whole family.
\end{rmk}
\begin{proof}
  There are two differences here from the results as they are typically presented
  in the probability literature, where the problem consists of matching two sets
  of random points of the same cardinality: first, the number of positive and
  negative points may not match; second, we are allowed to match points to the
  boundary as well as to points of the opposite orientation.\footnote{In fact, a
    somewhat similar, but more complicated modification was studied by Shor
    \cite{Shor}.}  We briefly explain how to modify the original proofs to deal
  with this.

  Clearly, the possibility of matching to the boundary cannot make the upper
  bounds worse.  Let's say without loss of generality there are more positive
  points.  To obtain the upper bound for $n \geq 2$, we may simply ignore some
  arbitrary set of ``extra'' positive points, matching all the others first.  By
  the central limit theorem, the expected number of extra points is
  $O(\sqrt{N})$, so the extra mass generated by matching them all to the boundary
  of the cube does not change the asymptotic answer.

  For the lower bound in the case $n=2$, we use the same stratagem of ignoring
  the ``extra'' points to create a new cycle $Z'$ with an equal number of
  positive and negative points.  From the original proof in \cite{AKT}, we know
  that there is a 1-Lipschitz function $f:[0,1]^2 \to \mathbb{R}$ which is zero
  on $\partial[0,1]^2$ and such that $\int_{Z'} f \geq c\sqrt{N\log N}$ with high
  probability.  Since with high probability the number of extra points is
  $<<\sqrt{N\log N}$, and the values of $f$ lie between $-1/2$ and $1/2$, we also
  know that $\int_Z f \geq c\sqrt{N\log N}$ with high probability.

  The lower bound in the case $n \geq 3$ is easy to see: conditional on any
  distribution of the positive points, most negative points will be at distance
  $\geq cN^{-1/n}$ from every positive point and the boundary, where $c>0$ is a
  constant depending on $n$.

  In the case $n=1$, the filling is unique up to a constant: the unique filling
  $F$ supported away from zero has density $\int_0^x Z$ at $x \in [0,1]$.  We use
  arguments found in \cite[\S3]{BobLMem} to give estimates on
  $\mathbb{E}(\mass F)$.

  The upper bound is a simple calculation:
  \begin{align*}
    \mathbb{E}(\mass F) &= \int_0^1
    \mathbb{E}\bigl(\bigl\lvert\textstyle{\int_0^x Z}\bigr\rvert\bigr) dx \\
    &\leq \int_0^1 \sqrt{\Var(\textstyle{\int_0^x Z})}dx=\frac{\sqrt{N}}{2}.
  \end{align*}
  The lower bound comes from the following classical fact, found in
  \cite{BobLMem} as Lemma 3.4:
  \begin{lem}
    Given independent mean zero random variables $\xi_1,\ldots,\xi_N$,
    \[\mathbb{E}\biggl(\biggl\lvert \sum_{k=1}^N \xi_k\biggr\rvert\biggr) \geq
    \frac{1}{2\sqrt 2}\mathbb{E}\biggl(\biggl(
    \sum_{k=1}^N \xi_k^2\biggr)^{1/2}\biggr).\]
  \end{lem}
  Let $(X_k,\sigma_k) \in [0,1] \times \{+1,-1\}$ be the $k$th chosen point.
  Then applying the lemma to $\xi_k=\sigma_k \chi_{\{X_k \leq x\}}$, we get
  \begin{align*}
    \mathbb{E}\bigl(\bigl\lvert\textstyle{\int_0^x Z}\bigr\rvert\bigr)
    &\geq \frac{1}{2\sqrt{2}}\mathbb{E}\biggl(\biggl(
    \sum_{k=1}^N \xi_k^2\biggr)^{1/2}\biggr) \\
    &\geq \frac{1}{2\sqrt{2}}\biggl(
    \sum_{k=1}^N (\mathbb{E}(|\xi_k|))^2\biggr)^{1/2}=\frac{1}{2\sqrt 2}\sqrt{N}x,
  \end{align*}
  and therefore $\mathbb{E}(\mass F) \geq \sqrt{N/32}$.
\end{proof}

%% Finally, we show a version of the theorem for manifolds other than the cube.
%% \begin{thm} \label{AKT:mfld}
%%   Let $M$ be a manifold with a free involution $\ph:M \to M$.  Generate a cycle
%%   $Z$ of mass $N$ in $C_0(M)$ by selecting $N$ i.i.d., uniformly distributed
%%   points $x_1,\ldots,x_N$ in $M$ and setting
%%   \[Z=\sum_{i=1}^N ([x_N]-[\ph(x_N)]).\]  Then
%%   the mass of an optimal filling of $Z$ is, with high probability,
%%   \[\left\{\begin{array}{l l}
%%   \Theta(\sqrt{N}) & \text{if }n=1 \\
%%   \Theta(\sqrt{N\log N}) & \text{if }n=2 \\
%%   \Theta(N^{(n-1)/n}) & \text{if }n \geq 3. \\
%%   \end{array}\right.\]
%% \end{thm}
%% \begin{proof}
%%   We use Theorem \ref{AKT:boundary}.  For the lower bound, choose a disk
%%   $B \subset M$ such that $\ph(B) \cap B=\emptyset$.  Any filling of $Z$ in $M$
%%   restricts to a filling of $Z \cap B$ as a cycle in $(B,\partial B)$, which
%%   consists of i.i.d.\ positive and negative points whose total number is at least
%%   $N\vol B/\vol M$ with high probability.

%%   By a theorem of Moser \cite{Moser}, this disk is diffeomorphic to
%%   $\lambda[0,1]^N$, for some scale factor $\lambda$, via a diffeomorphism
%%   preserving the volume form and therefore the uniform probability distribution.
%%   Thus by Theorem \ref{AKT:boundary}, we get our lower bound.
%% \end{proof}

\section{Proof of the upper bound} \label{S:upper}

To prove the upper bound in Theorems \ref{sphere} and \ref{cube}, we will use
Stokes' theorem; that is, we use the fact that for a cycle $Z \in C_k(M,A)$,
\begin{equation} \label{eqn:sup}
  FV(Z)=\sup\left\{\textstyle{\int_Z} \alpha : \alpha \in \Omega^k(M,A)
  \text{ such that }\lVert d\alpha \rVert_\infty=1\right\}.
\end{equation}
In fact, since $Z$ is a cycle, $\int_Z \alpha$ only depends on $\omega=d\alpha$.
To bound this quantity, we first note that any $\omega \in \Omega^{k+1}([0,1]^n)$
can be decomposed into a sum of ``basic'' forms of the form
\[\omega_I(x)dx_{i_1} \wedge \cdots \wedge dx_{i_{k+1}},\]
where $\omega_I$ is a function $\mathbb{R}^n \to \mathbb{R}$, for each subset
$\{i_1,\ldots,i_{k+1}\} \subset \{1,\ldots,n\}$.
\begin{lem} \label{lem:coIP}
  For any exact form $\omega \in \Omega^{k+1}([0,1]^n,\partial [0,1]^n)$
  (resp., $\omega \in \Omega^{k+1}([0,1]^n)$), there is
  a form $\alpha \in \Omega^k([0,1]^n,\partial [0,1]^n)$ (resp.,
  $\alpha \in \Omega^k([0,1]^n)$) given by
  \[\alpha=\sum_{\substack{I \subset [n]\\|I|=k}} \alpha_I(x)dx_I,\]
  such that $d\alpha=\omega$, and for each $I$,
  $\lVert \alpha_I \rVert_{\Lip}=\lVert d\alpha_I \rVert_\infty
  \leq C_{n,k}\lVert\omega\rVert_\infty$.
\end{lem}
\begin{proof}
  We prove this by induction on $n$ and $k$, keeping $n-k$ constant.  In the base
  case $k=0$, we can take the function $\alpha$ to be the fiberwise integral
  $\int_0^1 \omega$ along one of the coordinates.

  To do the inductive step in the relative case, we follow the usual proof of the
  Poincar\'e lemma with compact support, following \cite[\S1.4]{BottTu}.  Fix a
  smooth bump function $\epsi:[0,1] \to [0,1]$ which is 0 near 0 and 1 near 1.
  By applying the lemma one dimension lower, we get a form
  $\eta \in \Omega^{k-1}([0,1]^{n-1},\partial [0,1]^{n-1})$ with
  $d\eta=\int_0^1 \omega$ and
  $\lVert \eta_I \rVert_{\Lip} \leq C_{n-1,k-1}\lVert\omega\rVert_\infty$. Then
  $\omega=d\alpha$ for
  \[\alpha={\textstyle \int_0^t \omega}
  -\epsi(x_n)\pi^*({\textstyle \int_0^1\omega})
  -d\epsi(x_n) \wedge \pi^*\eta,\]
  where $\pi$ is the projection to the $(n-1)$-cube along $x_n$.  Notice that
  \[\alpha_I=\begin{cases}
  -{\displaystyle\frac{d\epsi}{dx}}\eta_{I \setminus \{n\}} & \text{if }n \in I \\
  \int_0^t \omega_{I \cup \{n\}}-\epsi(x_1)\pi^*(\int_0^1 \omega_{I \cup \{n\}}) &
  \text{otherwise.}
  \end{cases}\]
  This gives us a bound on each $\lVert\alpha_I\rVert_{\Lip}$ in terms of the
  $\lVert\omega_I\rVert_{\Lip}$ and $\lVert\eta_I\rVert_{\Lip}$ as well as the
  derivatives of $\epsi$.

  For the non-relative version, we follow the same proof, mutatis mutandis,
  taking
  \[\alpha={\textstyle \int_0^t \omega}-\pi^*\eta. \qedhere\]
\end{proof}

Here and in the next section, by a \emph{random $k$-cycle} we mean a random
variable taking values in $k$-cycles, that is, a measure on the space of
$k$-cycles.  We follow the convention, common in probability theory, of blurring
the distinction between a measure on a set of objects and an object randomly
drawn from that measure.
\begin{thm} \label{upperE}
  Let $Z$ be a random $k$-cycle in $([0,1]^n,\partial[0,1]^n)$ or in $[0,1]^n$
  which satisfies the condition that for some $M>0$,
  \begin{equation} \label{E-condition}
    \mathbb{E}(FV(Z \cap P)) \leq M
  \end{equation}
  for almost all $(n-k)$-planes $P$ parallel to one of the coordinate
  $(n-k)$-planes.  Then
  \begin{equation}
    \mathbb{E}(FV(Z)) \leq {n \choose k}C_{n,k}M,
  \end{equation}
  where $C_{n,k}$ is the constant from Lemma \ref{lem:coIP}.
  %%
  %% Moreover, suppose that in addition for every $P$,
  %% \begin{equation} \label{concentration-condition}
  %%   \mathbb{E}[\max\{FV(Z \cap P)-\mathbb{E}(FV(Z \cap P)),0\}]=E<\infty.
  %% \end{equation}
  %% Then for every $r>1$,
  %% \begin{equation} \label{concentration-result}
  %%   \mathbb{P}\left[FV(Z) \geq {n \choose k}C_{n,k}M+rE\right] \leq 1/r.
  %% \end{equation}
\end{thm}
\begin{proof}
  For a $k$-form $\alpha$, $\int_Z \alpha$ depends only on $d\alpha$.  Therefore
  to estimate \eqref{eqn:sup} it is enough to show that for any $(k+1)$-form
  $\omega$ with $\lVert\omega\rVert_\infty=1$, there is a $k$-form $\alpha$ such
  that $d\alpha=\omega$ and $\int_Z \alpha \leq C_{n,k}M$.

  By Lemma \ref{lem:coIP}, we can choose
  \[\alpha=\sum_{\substack{I \subset [n]\\|I|=k}} \alpha_I(x)dx_I\]
  such that $\lVert d\alpha_I \rVert_\infty \leq C_{n,k}$ for every $I$.  Then for
  $\alpha$ ranging over all these choices of antidifferentials,
  \begin{align*}
    FV(Z) &= \sup_\alpha \int_Z \alpha
    = \sup_\alpha \sum_{\substack{I \subset [n]\\|I|=k}} \int_{[0,1]^{I^c}}
    \int_{Z \cap P_u} \alpha_Idu \\
    &\leq \sum_{\substack{I \subset [n]\\|I|=k}} \int_{[0,1]^{I^c}}
    \left(\sup_\alpha\int_{Z \cap P_u} \alpha_I\right)du 
    \leq \sum_{\substack{I \subset [n]\\|I|=k}}
    \int_{[0,1]^{I^c}} C_{n,k}FV(Z \cap P_u).
  \end{align*}
  By linearity of expectation,
  \[\mathbb{E}(FV(Z)) \leq \sum_{\substack{I \subset [n]\\|I|=k}} \int_{[0,1]^{I^c}}
  C_{n,k}\mathbb{E}(FV(Z \cap P_u)) \leq {n \choose k}C_{n,k}M. \qedhere\]
  %% proving the first statement.  For the second statement,
  %% \eqref{concentration-condition} similarly implies that
  %% \[\mathbb{E}[\max\{FV(Z)-\mathbb{E}(FV(Z)),0\}] \leq {n \choose k}C_{n,k}E.\]
  %% Then \eqref{concentration-result} follows by a standard measure-theoretic
  %% argument.
\end{proof}
\begin{proof}[Proof of Theorem \ref{cube}, upper bound.]
  Let $Z$ be a cycle in $C_k([0,1]^n,\partial[0,1]^n)$ obtained by sampling $N$
  i.i.d.\ planes from a distribution on the set of oriented $k$-planes which
  intersect nontrivially with $[0,1]^n$, such that the distribution is uniform
  (with respect to Lebesgue measure on the corresponding polytope in
  $\mathbb{R}^{n-k}$) on each set of parallel planes.

  This condition clearly implies that for every coordinate $(n-k)$-plane $P$,
  $Z \cap P$ consists of at most $N$ i.i.d.\ positive and negative points with
  probability 1.  Then Theorem \ref{AKT:boundary} implies that
  \eqref{E-condition} holds for $Z$ with $M=C_{n-k}\AKT_{n-k}(N)$.%, and Theorem
%  \ref{concentration} implies that \eqref{concentration-condition} also holds
%  with $E=$ %%fill in this blank
%  .  This shows that
  \[\mathbb{E}(FV(Z)) \leq 2{n \choose k}C_{n,k}C_n\AKT_{n-k}(N). \qedhere\]
%  with high probability.
\end{proof}
\begin{proof}[Proof of Theorem \ref{sphere}, upper bound.]
  We use the fact that the transverse intersection of an oriented great
  $k$-sphere with an oriented great $(n-k)$-sphere is a pair of antipodal points
  with opposite orientations.  Therefore, if $Z$ is a cycle obtained by sampling
  $N$ oriented great $k$-spheres independently from the uniform distribution,
  then for any great $(n-k)$-sphere $P$, with probability 1
  \[Z \cap P=\sum_{i=1}^N [x_i]-[-x_i]\]
  where the $x_i$ are i.i.d.\ uniform points on $S^{n-k}$.

  Consider $S^n$ as a subset of $\mathbb{R}^{n+1}$, with standard unit basis
  vectors $e_0,\ldots,e_n$.  Let $K_i^\pm$ be the preimage of the cube $[-R,R]^n$
  under central projection (that is, projection along lines through the origin)
  to the plane $x_i=\pm 1$.  If $R$ is large enough, the interiors of the
  $K_i^\pm$ cover $S^n$.  Each $K_i^+$ is disjoint from its antipodal set
  $K_i^-$; therefore for any great $(n-k)$-sphere $P$, with probability 1
  $Z \cap P \cap K_i^\pm$ consists of i.i.d.\ uniform points.  By Remark
  \ref{AKT:diffeo},% with high probability
  \[\mathbb{E}(FV(Z \cap P \cap K_i^\pm)) \leq C_{n,k}\AKT_{n-k}(N),\]
  where $Z \cap P \cap K_i^\pm$ is considered as a cycle
  in $C_0(P \cap K_i^\pm, \partial(P \cap K_i^\pm))$.

  Note that central projection sends great spheres to hyperplanes.  Therefore, by
  Theorem \ref{upperE}, for each $i$ and sign,
  \begin{equation} \label{on-patch}
    \mathbb{E}(FV(Z \cap K_i^\pm)) \leq C_{n,k}\AKT_{n-k}(N).% \text{ with high probability.}
  \end{equation}

  Set a partition of unity $\{\ph_i^\pm\}$ subordinate to $\{K_i^\pm\}$ which is
  invariant with respect to the involution, that is, such that
  $\ph_i^+(x)=\ph_i^-(-x)$.  To prove the theorem, it is enough, given a $k$-form
  $\omega \in C_{k+1}(S^n)$ with $\lVert\omega\rVert_\infty=1$, to show that for
  some (and therefore every) $\alpha \in C_k(S^n)$ with $d\alpha=\omega$,
  \[\int_Z \alpha \leq C_{n,k}\AKT_{n-k}(N).\]
  But note that
  \[\int_Z \alpha=\sum_{i=0}^n \Bigl(\int_{Z \cap K_i^+} \ph_i^+\alpha+
  \int_{Z \cap K_i^-} \ph_i^-\alpha\Bigr).\]
  Therefore it suffices to find an antidifferential and a bound separately for
  each $\ph_i^\pm \omega$.  Therefore \eqref{on-patch} suffices to prove the
  theorem.
\end{proof}

\section{Proof of the lower bound} \label{S:lower}

\begin{thm} \label{lowerE}
  Let $Z$ be a random Lipschitz $k$-cycle in $([0,1]^n,\partial[0,1]^n)$ such
  that for almost every $k$-plane
  \[P_{\vec x}=\{(x_1,\ldots,x_k)\} \times [0,1]^{n-k} \subset [0,1]^n,\]
  the slice $Z \cap P_{\vec x}$ satisfies
  \[\mathbb{E}(FV(Z \cap P_{\vec x})) \geq p(\vec x)\]
  where $p:[0,1]^k \to [0,\infty)$ is an $L^1$ function.  Then
  \[\mathbb{E}(FV(Z)) \geq \int_{[0,1]^k} p(\vec x)d\vec x.\]
\end{thm}
\begin{proof}
  Let $U$ be a normal current filling $Z$ such that $\mass(U) \leq FV(Z)+\epsi$,
  for any $\epsi>0$.  Then for almost all $P_{\vec x}$, there is a slice
  $U \cap P_{\vec x}$ which fills $Z \cap P_{\vec x}$, and
  \[\mass(U \cap P_{\vec x}) \geq FV(Z \cap P_{\vec x}).\]
  By Proposition \ref{slice-coarea},
  \[FV(Z)+\epsi \geq \mass(U) \geq \int_{[0,1]^k} \mass(U \cap P_{\vec x})d\vec x
  \geq \int_{[0,1]^k} FV(Z \cap P_{\vec x})d\vec x.\]
  Since this is true for every $\epsi>0$, and by linearity of expectation,
  \[\mathbb{E}(FV(Z)) \geq \int_{[0,1]^k} \mathbb{E}(FV(Z \cap P_{\vec x}))d\vec x
  \geq \int_{[0,1]^k} p(\vec x)d\vec x.\qedhere\]
\end{proof}
\begin{proof}[Proof of Theorem \ref{cube}, lower bound.]
  Let $Z$ be a cycle in $C_k([0,1]^n,\partial[0,1]^n)$ obtained by sampling $N$
  i.i.d.\ planes from a distribution on the set of oriented $k$-planes which
  intersect nontrivially with $[0,1]^n$, such that the distribution is uniform
  (with respect to Lebesgue measure on the corresponding polytope in
  $\mathbb{R}^{n-k}$) on each set of parallel planes.  Assume, perhaps by
  permuting coordinates, that this distribution is not concentrated on planes of
  the form
  \[P_{\vec x}=(x_1,\ldots,x_k) \times \mathbb{R}^{n-k}.\]

  As in the proof of the upper bound, it follows that for every $P_{\vec x}$,
  $Z \cap P_{\vec x}$ consists of i.i.d.\ positive and negative points with
  probability 1.  Moreover, the probability of a random plane $P$ intersecting
  $P_{\vec x}$ inside $[0,1]^n$ depends only on the direction of $P$ and not on
  $\vec x$.  Thus
  \[\mathbb{E}(\mass(Z \cap P_{\vec x})) \geq cN,\]
  where $c$ depends on the distribution but not on $\vec x$.%  For every $\vec x$,
%  $\mass(Z \cap P_{\vec x})$ is close to this expected value with high
  %  probability.
  Thus by Theorems \ref{AKT:boundary} and \ref{lowerE},
  \[\mathbb{E}(FV(Z)) \geq \frac{1}{2}C_{n-k}\AKT_{n-k}(cN). \qedhere\]
\end{proof}
\begin{proof}[Proof of Theorem \ref{sphere}, lower bound.]
  Let $Z$ be a cycle in $C_k(S^n)$ obtained by sampling $N$ independent uniformly
  distributed great $k$-spheres.  It suffices to show that for some compact
  submanifold $K \subset S^n$,
  \[FV(Z \cap K) \geq C_{n,k}\AKT_{n-k}(N)\]
  where $Z \cap K$ is considered as a cycle in $C_k(K,\partial K)$.

  Recall that for any great $(n-k)$-sphere $P$, with probability 1
  \[Z \cap P=\sum_{i=1}^N [x_i]-[-x_i]\]
  where the $x_i$ are i.i.d.\ uniform points on $S^{n-k}$.  Let $T \subset S^n$ be
  a great $k$-sphere and $T'$ the $(n-k)$-sphere consisting of points farthest
  from $T$.  We use $N_\epsi(U)$ to indicate the $\epsi$-neighborhood of the set
  $U$; then $S^n \setminus N_{\pi/4}(T')=\overline{N_{\pi/4}(T)}$ deformation
  retracts to $T$ along the orthogonal retraction
  $\rho:\overline{N_{\pi/4}(T)} \to T$.  We let $K=\rho^{-1}(K')$ where $K'$ is
  some closed ball in $T$ which does not include any point and its antipode.

  Notice that $K$ is foliated by equal-volume patches of great $(n-k)$-spheres
  $P_u$ which retract to points $u \in K'$, and also does not include any point
  and its antipode.  By Remark \ref{AKT:diffeo}, there is a bilipschitz
  diffeomorphism from $K$ to $[0,1]^n$ which sends each $P_u$ to a plane of the
  form
  \[(x_1,\ldots,x_k) \times \mathbb{R}^{n-k}\]
  in a volume-preserving way (up to a constant).  Therefore, for each $u \in K'$,
  \[\mathbb{E}(FV(Z \cap P_u \cap K)) \geq c\AKT_{n-k}(cN),\]
  and applying Theorem \ref{lowerE}, we obtain the result.
\end{proof}

\section{Proof of Theorem \ref*{knot}} \label{S:knot}

Before we prove Theorem \ref{knot}, we must give a more precise statement.

By an \emph{oriented $k$-pseudomanifold} we mean a $k$-dimensional simplicial
complex $M$ with the following properties:
\begin{itemize}
\item It is \emph{pure}, i.e.\ every simplex is contained in an $k$-dimensional
  simplex.
\item Every $k$-simplex comes with an orientation such that the sum of all the
  oriented $k$-simplices is a cycle in $C_k(M)$.
\end{itemize}
Note that this is considerably wider than the usual definition: it is just enough
so that if $M$ is equipped with the standard simplexwise metric, any Lipschitz
map from $M$ to a metric space $X$ defines a Lipschitz $k$-cycle in $X$.

We say $M$ has \emph{geometry bounded by} $L$ if every $k$-simplex intersects at
most $L$ others.

With these definitions, we restate Theorem \ref{knot}:
\begin{thm*}
  Let $M$ be an oriented $k$-pseudomanifold with $N$ vertices and geometry
  bounded by $L$.  Let $Z$ be a $k$-cycle on $[0,1]^n$ obtained by sending each
  vertex of $M$ to a uniformly random point in $[0,1]^n$ and extending linearly.
  Then there are constants $C>c>0$ depending on $n$ and $k$ such that
  \begin{equation} \label{AKTstats:knot}
    cL^{-1}\AKT_{n-k}(N)<\mathbb{E}(FV(Z))<CL\AKT_{n-k}(N).
  \end{equation}
\end{thm*}
The concentration result will be proved in the next section.

As with Theorem \ref{cube}, the proof is a direct application of Theorems
\ref{upperE} and \ref{lowerE}.  To apply these theorems, we need to understand
the filling volumes of slices of $Z$, which is more complicated in this case
because while the points are identically distributed, they are not entirely
independent.  We establish the upper bound in Lemma \ref{knot:upper}; this
depends only on the fact that every point is independent of all but a constant
number of others.  For the lower bound in Theorem \ref{knot:lower}, the argument
is more subtle: even if every point is correlated with only one other, such pairs
could be very close and have opposite signs; then the least filling would be much
smaller than for independent points.  Accordingly, we have to show that most
correlated points are still far apart.  Together, these two bounds complete the
proof.

In this section as before, fix the notation
\[P_{\vec x}=\{\vec x\} \times [0,1]^{n-k} \subset [0,1]^n, \qquad
\vec x \in [0,1]^k.\]
We start by analyzing the slice $Z \cap P_{\vec x}$.
\begin{lem} \label{slice-properties}
  Let $\vec x \in [0,1]^k$.  Then the slice $Z \cap P_{\vec x}$ is the sum of $N$
  random $0$-chains $\zeta_1,\ldots,\zeta_N$ which are identically distributed on
  $\{\pm[y] : y \in [0,1]^{n-k}\} \cup \{0\}$ according to a distribution
  $\mu_{\vec x}$ depending on $k$ and $\vec x$.  Moreover:
  \begin{enumerate}[(i)]
  \item $\mu_{\vec x}$ is invariant with respect to the involution sending a chain
    $\zeta$ to $-\zeta$;
  \item $\mu_{\vec x} \leq C(n,k)\mu_{\text{Lebesgue}}$ on $[0,1]^{n-k}$.
  \item Each $\zeta_i$ is independent of all but at most $L$ other $\zeta_j$.
  \end{enumerate}
\end{lem}
Since the distribution of $Z$ is invariant under permuting coordinates, this
holds for any $(n-k)$-dimensional slice in a coordinate direction.
\begin{proof}
  The distribution in question is the intersection of a random linear $k$-simplex
  in $[0,1]^n$ with $P_{\vec x}$.  Property (i) is obvious from this, and (iii)
  follows since a pair of $\zeta_i$ are independent whenever the two
  corresponding simplices do not intersect.  To see (ii), consider the function
  \[F_{\vec x}:\bigl(([0,1]^n)^{k+1},\mu_{\text{Lebesgue}}\bigr) \to
  \bigl(\{\pm[y] : y \in [0,1]^{n-k}\} \cup \{0\},\mu_{\vec x}\bigr)\]
  sending each linear $k$-simplex to its intersection with $P_{\vec x}$.  This
  function is measure-preserving by definition, and its restriction to
  \[K=F_{\vec x}^{-1}\{[y] : y \in [0,1]^{n-k}\}\]
  is 1-Lipschitz.  Therefore, by the coarea formula, the density function of
  $\mu_{\vec x}$ is given by the $[n(k+1)-(n-k)]$-dimensional Hausdorff measure of
  point preimages.  Thus it is enough to bound $H_{(n+1)k}(F_{\vec x}^{-1}(\vec y))$
  for each $\vec y$.

  Let $T$ be the set of linear $k$-simplices with vertices in
  $[-1,1]^n$ which pass through $\vec 0$, and let
  \[T_{\vec x}=(T+(\vec x,\vec 0)) \cap ([0,1]^k \times [-1,1]^{n-k})^{k+1}.\]
  (Here each vertex is translated by $(\vec x,\vec 0)$.)  Notice that
  \[F_{\vec x}^{-1}(\vec y) \subset T_{\vec x}+(\vec 0,\vec y).\]
  All these translates are disjoint and their union is a subset of
  $([0,1]^k \times [-1,2]^{n-k})^k$.  Therefore, again by the coarea formula,
  \[H_{(n+1)k}(T_{\vec x}) \leq 3^{(k+1)(n-k)}.\]
  This completes the proof of (ii).
\end{proof}
Condition (iii) gives a dependency graph of degree $\leq L$ between the
$\zeta_i$.  This graph has an $(L+1)$-coloring, giving a partition of
$\{0,\ldots,N\}$ into $L+1$ disjoint subsets $I_1,\ldots,I_n$ such that for
$i \in I_j$, the $\zeta_i$ are i.i.d.

\subsection{Upper bound}
This lemma gives the upper bound to plug into Theorem \ref{upperE}:
\begin{lem} \label{knot:upper}
  For every $\vec x \in [0,1]^k$,
  \[\mathbb{E}(FV(Z \cap P_{\vec x})) \leq (L+1)C_{n,k}\AKT_{n-k}(N).\]
\end{lem}
\begin{proof}
  We estimate $\mathbb{E}(FV(Z \cap P_{\vec x}))$ by separately considering each
  summand
  \[Z(\vec x,j)=\sum_{i \in I_j} \zeta_i,\qquad j=1,\ldots,L+1.\]
  These summands consist of i.i.d.~negative and positive points.

  Write $\nu_{\vec x}$ for the probability measure on $[0,1]^{n-k}$ given by
  \[\nu_{\vec x}(A)=\frac{\mu_{\vec x}\{[\vec y] : \vec y \in A\}}
       {\mu_{\vec x}\{[\vec y] : \vec y \in [0,1]^{n-k}\}}.\]
  If $\zeta$ is a random 0-cycle in $[0,1]^{n-k}$ with $N$ positive and $N$
  negative points distributed according to $\nu_{\vec x}$, then the AKT upper
  bound holds for $\zeta$: by \cite[equation (12)]{BobL}, for some constant
  $C_{n-k}$ independent of the measure,
  \begin{equation} \label{uniform-AKT}
    \mathbb{E}(FV(\zeta)) \leq C_{n-k}\AKT_{n-k}(N).
  \end{equation}

  To reduce to this situation, we note that while $Z \cap P_{\vec x}$ is a cycle
  (and therefore has an equal number of negative and positive points),
  $Z(\vec x,j)$ may not be.  We produce cycles $\tilde Z(\vec x,j)$ for
  $j=1,\ldots,L+1$ by adding up to $N$ additional i.i.d.~points distributed
  according to $\nu_{\vec x}$.  We add each point to $\tilde Z(\vec x,j)$ for two
  different $j$, with opposite signs, so that
  \[\sum_{j=1}^{L+1} \tilde Z(\vec x,j)=\sum_{j=1}^{L+1} Z(\vec x,j)=
  Z \cap P_{\vec x}.\]
  Each $\tilde Z(\vec x,j)$ is a 0-cycle consisting of at most $N$ positive and
  $N$ negative i.i.d.~points.  Therefore, by \eqref{uniform-AKT},
  \[\mathbb{E}(FV(Z \cap P_{\vec x})) \leq (L+1)C_{n,k}\AKT_{n-k}(N).\qedhere\]
\end{proof}

\subsection{Lower bound}
For the lower bound, we begin by showing that correlated points in
$Z \cap P_{\vec x}$ are usually not very close; this relationship is summarized in
Lemma \ref{correlation}.  We use this as an ingredient in the proof of Theorem
\ref{knot:lower}, which retraces some of the steps in the original proof of the
AKT theorem.

In many places, we refer to the function $F_{\vec x}$ and set $T_{\vec x}$ defined
in the proof of Lemma \ref{slice-properties}.  Also, given a subset
$A \subseteq [0,1]^{n-k}$, we write
\[[A]=\{[y]: y \in A\}\qquad\text{and}\qquad{\pm[A]}=\{\pm[y]: y \in A\}\]
for the respective sets of $0$-dimensional chains.
\begin{lem} \label{correlation}
  Assume that $\vec x \in [1/4,3/4]^k$.  Let $0<\ell<1/2$ and let $\zeta$ and
  $\zeta'$ be random variables whose values are the intersections with
  $P_{\vec x}$ of two $k$-simplices $\Delta$ and $\Delta'$ of $M$ whose vertices,
  some of which are shared, are chosen uniformly at random from $[0,1]^n$.  Let
  $Q \subset [1/4,3/4]^{n-k}$ be a cube of side length $\ell$.  Then
  \begin{equation} \label{eqn:correlation}
    \mathbb{P}\bigl[\zeta' \in \pm[Q] \mid \zeta \in [Q]\bigr]
    \leq C_{n,k}\sqrt{\ell}.
    %\begin{cases}
    %  C_{n,k}\ell^{2/3} & \text{if }n-k=1, \\
    %  C_{n,k}\sqrt{\ell} & \text{otherwise}.
    %\end{cases}
  \end{equation}
\end{lem}
\begin{rmk}
  A more careful analysis based on the same principle shows that
  \[\mathbb{P}\bigl[\zeta' \in \pm[Q] \mid \zeta \in [Q]\bigr] \leq
  \begin{cases}
    C_{n,k}\ell \lvert\log \ell\rvert & \text{if }n-k=1, \\
    C_{n,k}\ell & \text{otherwise}.
  \end{cases}\]
\end{rmk}
\begin{proof}
  The idea is this: suppose that $\Delta$ and $\Delta'$ share a $(k-1)$-face
  $\Delta_0$, and let $w$ and $w'$ be the non-shared vertices of $\Delta$ and
  $\Delta'$.  If the intersections of $\Delta$ and $\Delta'$ with $P_{\vec x}$ are
  close to each other, then either the angle between $\Delta$ and $\Delta'$ is
  small (forcing $w'$ to be close to the $k$-plane containing $\Delta$), or
  $\Delta_0$ is close to $P_{\vec x}$.  We would like to show that neither of
  these happens too often.

  We will do this case in detail; the analysis when $\Delta$ and $\Delta'$ share
  a lower-dimensional face is similar.

  First, we show that $\Delta_0$ is not very often too close to $P_{\vec x}$.
  This is easy to establish globally; the tricky bit is showing that it's true
  after conditioning on $\zeta$ being supported in $Q$.  Let
  $\rho_{\vec x}(\Delta_0)$ be the distance from $\Delta_0$ to the hyperplane
  $\{\vec x\} \times \mathbb{R}^{n-k}$.  The following lemma will be proved
  later.
  \begin{lem} \label{not-too-close}
    Let $\vec x \in [1/4,3/4]^k$ and let $A \subseteq [1/4,3/4]^{n-k}$ be a set of
    positive measure.  Let $\Delta$ be a simplex with vertices chosen uniformly
    at random from $[0,1]^n$, and let $\Delta_0$ be its $0$th face.  Then
    \[\mathbb{P}\bigl[\rho_{\vec x}(\Delta_0) \leq r \mid
      F_{\vec x}(\Delta) \in [A]\bigr] \leq C_{n,k}r.\]
  \end{lem}
  In particular,
  \begin{equation} \label{closer}
    \mathbb{P}\bigl[\rho_{\vec x}(\Delta_0) \leq \sqrt{\ell} \mid
      \zeta \in [Q]\bigr] \leq C_{n,k}\sqrt{\ell}.
  \end{equation}

  Now we must show that when $\rho_{\vec x}(\Delta_0)>\sqrt{\ell}$, $\zeta'$
  doesn't very often land near $\zeta$.  The point is that the difference depends
  on the angle between $\Delta$ and $\Delta'$, and the distribution of this angle
  is not too concentrated anywhere; this is true for any given $\Delta_0$.  Fix
  $\Delta_0$ with $\rho(\Delta_0) \geq \sqrt{\ell}$, and let
  $U(\Delta_0,Q) \subseteq [0,1]^n$ be the set of points $w'$ such that
  $\zeta' \in \pm[Q]$.  Note that $U(\Delta_0,\{z\})$ is contained in the
  intersection of a $k$-plane with $[0,1]^{n(k+1)}$ and hence its $k$-dimensional
  Hausdorff norm is at most some $C_{n,k}$.  Now we would like to use the coarea
  formula to integrate with respect to $z \in Q$.  For this we need the following
  estimate, to be proved later:
  \begin{lem} \label{Jacobian}
    Given a linear $k$-simplex $\Delta \in ([0,1]^n)^{k+1}$ such that at least one
    of its $(k-1)$-faces is at distance at least $r$ from $P_{\vec x}$,
    \[\sqrt{\det\bigl([(DF_{\vec x})_\Delta]^{T}(DF_{\vec x})_\Delta\bigr)}
    \geq C_{n,k}r^{n-k}.\]
  \end{lem}
  Then by the coarea formula, for fixed $\Delta_0$,
  \[\left(\frac{\ell}{k}\right)^{\frac{n-k}{2}}\vol(U(\Delta_0,Q)) \leq
  C_{n,k}\vol(Q)\]
  and therefore
  \[\vol(U(\Delta_0,Q)) \leq C_{n,k}\ell^{\frac{n-k}{2}}.\]
  Integrating this over the domain in $[0,1]^{kn}$ containing all values of
  $\Delta_0$ such that $\rho_{\vec x}(\Delta_0) \geq \sqrt{\ell}$, we see that
  \begin{equation} \label{farther}
    \mathbb{P}\bigl[\zeta' \in \pm[Q] \mid \zeta \in [Q],
      \rho(\Delta_0) \geq \sqrt{\ell}\bigr] \leq C_{n,k}\ell^{\frac{n-k}{2}}.
  \end{equation}
  Together, \eqref{closer} and \eqref{farther} imply \eqref{eqn:correlation}.
\end{proof}
Now we prove the lemmas.
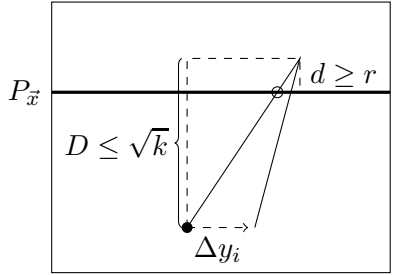
\begin{figure}
  \begin{tikzpicture}[scale=3]
    \draw (0,-0.2) -- (0,1) -- (1.5,1) -- (1.5,-0.2) -- cycle;
    \draw[very thick] (0,0.6) node[anchor=east] {$P_{\vec x}$} -- (1.5,0.6);
    \coordinate (face) at (1.1,0.75);
    \coordinate (pt) at (0.6,0);
    \coordinate (new-pt) at (0.9,0);
    \draw (new-pt) -- (face) -- (pt);
    \tikzstyle{dott}=[circle,scale=0.4];
    \node[dott,draw=black] (intersection) at (1,0.6) {};
    \draw[dashed,->] (pt) node[dott,fill=black]{} -- (0.87,0)
      node[midway,anchor=north] {$\Delta y_i$};
    \draw[dashed] (pt) -- +(0,0.75) -- (face) -- +(0,-0.15);
    \draw[snake=brace,raise snake=2pt]
    (pt) -- +(0,0.75) node[midway,anchor=east,inner sep=7pt] {$D \leq \sqrt{k}$};
    \node[right] at (1.1,0.675) {$d \geq r$};
  \end{tikzpicture}
  \caption{When a vertex moves by $\Delta y_i$, $\Delta \cap P_{\vec x}$ moves by
    $(d/D)\Delta y_i$.} \label{fig}
\end{figure}
\begin{proof}[Proof of Lemma \ref{Jacobian}]
  From Figure \ref{fig}, we see that for every $1 \leq i \leq n-k$, there
  is a unit vector $\vec v$ such that $(DF_{\vec x})_\Delta(\vec v)=c\vec e_i$, for
  $c>r/\sqrt{k}$.  Therefore
  \[\sqrt{\det\bigl([(DF_{\vec x})_\Delta]^{T} (DF_{\vec x})_\Delta\bigr)}
  \geq \left(\frac{r}{\sqrt{k}}\right)^{n-k}. \qedhere\]
\end{proof}
\begin{proof}[Proof of Lemma \ref{not-too-close}]
  Let $X$ be the event that $\rho_{\vec x}(\Delta) \leq r$, and let $Y$ be the
  event that $\Delta \cap P_{\vec x} \in [Q]$.  Bayes' rule states that
  \[\mathbb{P}(X \mid Y)=\mathbb{P}(Y \mid X)\mathbb{P}(X)/\mathbb{P}(Y).\]
  We will bound $\mathbb{P}(X \mid Y)$ by giving upper bounds for
  $\mathbb{P}(Y \mid X)$ and $\mathbb{P}(X)$ and a lower bound for
  $\mathbb{P}(Y)$.

  We first show that
  \begin{equation} \label{P(Y|X)}
    \mathbb{P}(Y \mid X) \leq 3^{(n-k)(k+1)}\vol(A).
  \end{equation}
  In fact, in this inequality, $X$ may be any constraint on the first $k$
  coordinates of each vertex of $\Delta$ (note that $\rho_{\vec x}(\Delta)$
  depends only on those coordinates).

  We start by fixing the first $k$ coordinates of each vertex.  Given a simplex
  $\Delta$ with vertices in $\mathbb{R}^n$, let $\pi(\Delta)$ be the projection
  onto the first $k$ coordinates, and fix a simplex $\Delta_\pi$ with vertices in
  $[0,1]^k$.  Let $U \subset ([-1,1]^{n-k})^{k+1}$ contain the last $n-k$
  coordinates of simplices in $\pi^{-1}(\Delta_\pi)$ whose intersection with
  $P_{\vec x}$ is $[(\vec x,\vec 0)]$.  Write $U+A$ to mean the set of translates
  of simplices in $U$ by points in $A$.  Notice that
  \begin{enumerate}
  \item $U+[0,1]^{n-k} \subseteq ([-1,2]^{n-k})^{k+1}$,
  \item the volume of $U+A$ is proportional to $\vol(A)$, and
  \item $U+\{\vec y\}$ contains $F^{-1}([\vec y]) \cap \pi^{-1}(\Delta_\pi)$.
  \end{enumerate}
  Therefore, the
  probability that a given point of $([0,1]^{n-k})^{k+1}$ is contained in $U+A$ is
  at most $3^{(n-k)(k+1)}\vol(A)$, and this in turn bounds
  \[\vol(F^{-1}([A])\cap\pi^{-1}(\Delta_\pi)).\]
  Integrating over possible values of $\Delta_\pi$ gives \eqref{P(Y|X)}.

  Now we show that $\mathbb{P}(X) \leq C_{n,k}r$.  Choosing $k$ points in
  $\mathbb{R}^k$ uniformly at random induces a probability measure on the set of
  $(k-1)$-planes, whose density at a plane $P$ is proportional to
  $\vol(P \cap [0,1]^k)^k$.  This density is bounded above by some $C_{n,k}$, and
  so the set of planes whose distance from $\vec x$ is at most $r$ has measure
  $\leq C_{n,k}r$.

  Finally we must give a lower bound for $\mathbb{P}(Y)$, that is, the volume of
  \[\Sigma_Y=\{\Delta: F_{\vec x}(\Delta) \in [A]\} \subseteq [0,1]^{n(k+1)}.\]
  For this, we note that when $\vec x$ and $\vec y$ have all coordinates in
  $[1/4,3/4]$,
  \[F_{\vec x}^{-1}([\vec y]) \supseteq T_{\vec x} \cap
  ([-1/4,1/4]^n+\{(\vec x,\vec 0)\})+\{(\vec 0,\vec y)\}.\]
  This is a translate of a set which is independent of $\vec x$ and $\vec y$.
  Taking all the translates with respect to vectors in $A$, we see that
  $\Sigma_Y$ contains a set whose volume depends linearly on $\vol(A)$ and is
  easily seen to be positive.

  Thus overall we get $\mathbb{P}(X \mid Y) \leq C_{n,k}r$.
\end{proof}

Finally we have the tools we need to prove the lower bound for Theorem
\ref{knot}.
\begin{thm} \label{knot:lower}
  For every $\vec x \in [1/4,3/4]^k$,
  \[\mathbb{E}(FV(Z \cap P_{\vec x})) \geq c_{n,k}L^{-1}\AKT_{n-k}(N).\]
\end{thm}
\begin{proof}
  Write $d=n-k$.  We are trying to find a lower bound for the expected filling
  length of a certain distribution on $0$-cycles in $[0,1]^d$ which is defined as
  a sum of a large number of identically distributed points.  This is very
  similar to the original AKT theorem, in which the points are, in addition,
  independent and uniformly distributed.  In \cite{AKT}, the lower bound is
  proved using the dual definition of filling volume given in \eqref{eqn:sup},
  which we restate here for a $0$-cycle $Z_0$ in $[0,1]^d$:
  \[FV(Z_0)=\sup\bigl\{\textstyle{\int_{Z_0}} f:[0,1]^d \to \mathbb{R}
  \text{ such that }\Lip(f) \leq 1\bigr\}.\]
  For the case $d=2$, Ajtai, Koml\'os and Tusn\'ady construct a
  $\sqrt{N\log N}$-Lipschitz function $f$ whose integral over $Z_0$ is usually at
  least $cN\log N$.  The filling volume $FV(Z_0)$ is then bounded below by the
  ratio of the integral to the Lipschitz constant.

  When $d=2$, we construct $f$ the same way as in \cite{AKT}, but on a modified
  point set in order to overcome two issues:
  \begin{enumerate}
  \item The points of $Z \cap P_{\vec x}$ are not uniformly distributed.  We use a
    bilipschitz rescaling of the domain to make sure that they are.
  \item The points of $Z \cap P_{\vec x}$ are not independent.  We instead
    construct the function $f$ to have a large integral over a large independent
    subset, and then use Lemma \ref{correlation} to show that the integral over
    the remaining points is small.
  \end{enumerate}
  For $d=1$ and $d \geq 3$, we use the same reduction steps, but the function $f$
  is somewhat simpler.  We start by more precisely describing the common argument
  and then give the detailed construction of $f$ for each case.

  Let $\zeta_i$ be the chain-valued random variables corresponding to
  intersections of $k$-simplices of $Z$ with $P_{\vec x}$.  Recall that
  $\zeta_i=\pm [v_i]$ or $\{0\}$, with $v_i \in [0,1]^d$ identically distributed
  according to a density which is bounded below on $[1/4,3/4]^d$.  Moreover, this
  bound is uniform with respect to $\vec x \in [1/4,3/4]^k$.  Thus, following
  Remark \ref{AKT:diffeo}, there is a uniformly bilipschitz family of
  diffeomorphisms
  \[\ph_{\vec x}:[1/4,3/4]^d \to [0,1]^d\]
  which send this density to a constant times the standard volume form.  In
  particular, Lemma \ref{correlation} still holds for the $\zeta_i$ after
  applying the diffeomorphism.  We now write $\zeta_i$ for
  $\ph_{\vec x}(\zeta_i)$.

  In each case, consider an $(L+1)$-coloring $I_1,\ldots,I_{L+1}$ of the
  dependency graph between the $\zeta_i$, with the colored subsets ordered from
  largest to smallest.  Note that each of the $\zeta_i$ is correlated with
  $\zeta_{i'}$ for at most $L$ values of $i'$.

  We write $Z_j=\sum_{i \in I_j} \zeta_i$.  In each case, we show that $Z_1$ is
  hard to fill by constructing a Lipschitz function $f$ such that
  \[\int_{Z_1} f \geq \frac{c_{n,k}}{L}\AKT_{n-k}(N)\Lip f
  \text{ with high probability.}\]
  Then we show that $\mathbb{E}(\int_{Z_j} f)$ for each $j \neq 1$ is small enough
  that it does not affect the overall asymptotics.

  In each case, the function $f$ is constructed as a sum of simpler functions.
  Given a cube $Q \subseteq [0,1]^d$, let $\Delta_Q:[0,1]^d \to \mathbb{R}$ be
  the function supported on $Q$ whose graph is a symmetric pyramid with base $Q$
  and height 1.  We also write $\Delta^r_v$ for the cube in the lattice of side
  length $r$ which contains $v \in [0,1]^d$.  The function $f$ will consist of a
  sum of scaled copies of $\Delta_Q$, each reflecting the ``imbalance'' of
  positive and negative points on the cube $Q$.  The main difference between
  different dimensions is the scale of these cubes: for $d \geq 3$, the cubes are
  at the smallest scale (comparable to the average distance between neighboring
  points), for $d=1$ they are at the largest scale (comparable to 1), and for
  $d=2$ we use cubes at many scales, as in the original proof of \cite{AKT}.

  In each case, we construct $f$ by means of an auxiliary function
  \[g(x)=\sum_{i \in I_1} g_{\zeta_i}(x)\]
  (in the case $d=2$, each $\zeta_i$ is associated to many summands $g^r_{\zeta_i}$
  at different scales, which we consider separately).  This $g$ may not satisfy
  the desired upper bound on the Lipschitz constant, so we remove some of the
  summands where they are too concentrated to produce $f$.  Whenever $\zeta_i$ is
  independent from $\zeta_{i'}$,
  $\mathbb{E}\bigl(\int_{\zeta_{i'}} g_{\zeta_i}(x)\bigr)=0$, and so for every
  $j \neq 1$ we can write
  \[\bigl\lvert\mathbb{E}\bigl({\textstyle\int_{Z_j}} g\bigr)\bigr\rvert
  \leq \sum_{i \in I_j} \sum \{\bigl\lvert\mathbb{E}\bigl(
       {\textstyle\int_{\zeta_i}}g_{\zeta_{i'}}\bigr)\bigr\rvert
  \mid \zeta_{i'}\text{ is correlated with }\zeta_i\}.\]
  By Lemma \ref{correlation}, this correlation is not too high, and therefore we
  can bound each of the summands.  By the construction of $f$, this also bounds
  $\bigl\lvert\mathbb{E}\bigl(\int_{Z_j} f\bigr)\bigr\rvert$.

  If we tune everything correctly, we get that for each $j \neq 1$,
  \[\bigl\lvert\mathbb{E}\bigl({\textstyle\int_{Z_j}} f\bigr)\bigr\rvert
  \leq \frac{1}{2L}\mathbb{E}\bigl({\textstyle\int_{Z_1}} f\bigr),\]
  giving a lower bound on $\mathbb{E}\bigl(\int_Z f\bigr)$.

  \subsection*{Case $d \geq 3$}
  We split the cube $[0,1]^d$ into $\sim N$ subcubes of side length
  $r \approx N^{-1/d}$, and let
  \begin{align*}
    g(x) &= \sum_{i \in I_1} g_i(x)=
    \sum_{i \in I_1} \pm r\Delta^r_{v_i}(x)\text{ where }\zeta_i=\pm[v_i], \\
    f(x) &= \sum_{t_1,\ldots,t_d=0}^{r^{-1}-1}
    \sgn\Bigl(\int_{Z_1}\chi_{Q_r(t_1,\ldots,t_d)}\Bigr) r\Delta_{Q_r(t_1,\ldots,t_d)}(x),
  \end{align*}
  where $Q_r(t_1,\ldots,t_d)$ is the cube with side length $r$ whose vertex
  closest to the origin is $(rt_1,\ldots,rt_d)$.  Note that $f$ is $2$-Lipschitz.

  We can think of the number of points landing in each subcube as $N$ independent
  $\lambda=1$ Poisson processes which we stop once their sum reaches roughly
  \[\mathbb{P}(\zeta_i \cap [0,1]^d \neq 0)\lvert I_1 \rvert.\]
  By the law of large numbers, the stopping time will be very close to
  \[t=N^{-1}\mathbb{P}(\zeta_i=\pm[y], y \in [0,1]^d)\lvert I_1 \rvert\]
  and very nearly $Nte^{-t}$ of the subcubes will contain exactly one point.  Of
  these, with high probability, at least $(1/2) \cdot 3^{-d} \cdot Nte^{-t}$ will
  be contained in the middle third of their subcube.  Therefore,
  \[\int_{Z_1} f \geq \frac{c_{n,k}}{L} \cdot N^{\frac{d-1}{d}}
  \text{ with high probability}.\]
  On the other hand, given $j \neq 1$, $i \in I_j$, and $i' \in I_1$ such that
  $\zeta_i$ is correlated with $\zeta_{i'}$, Lemma \ref{correlation} tells us
  that
  \[\bigl\lvert\mathbb{E}\bigl(\textstyle{\int_{\zeta_i}} g_{i'}\bigr)\bigr\rvert
  \leq C_{n,k}r^{3/2}\]
  and therefore
  \[\bigl\lvert\mathbb{E}\bigl(\textstyle{\int_{Z_j}} f\bigr)\bigr\rvert \leq
  \sum_{i \in I_j} LC_{n,k}r^{3/2} \leq LC_{n,k}N^{\frac{d-1}{d}-\frac{1}{2d}}.\]
  Since this is small compared to $\int_{Z_1} f$, this shows that
  \[\mathbb{E}\Bigl(\int_{Z \cap P_{\vec x}} f\Bigr)
  \geq \frac{c_{n,k}}{L}N^{\frac{d-1}{d}}\Lip f\qquad\text{for large enough }N.\]

  \subsection*{Case $d=2$}
  This case is broadly similar, but we build the function $f$ in a more
  complicated way, following the original proof of \cite{AKT}.  For an integer
  $r$, let $Q^r_{st}$ be the square of side length $2^{-r}$ whose lower left
  corner is at $(s\cdot 2^{-r},t\cdot 2^{-r})$, and write
  $\Delta^r_{st}=\Delta_{Q^r_{st}}$.  We write
  \begin{align*}
    g(x,y) &= \sum_{r=1}^{0.1\log N} \sum_{s,t=0}^{2^r-1} g^r_{st}(x,y)=
    \sum_{r=1}^{0.1\log N} \sum_{s,t=0}^{2^r-1} \Delta^r_{st}(x,y)
    \int_{Z_1} \Delta^r_{st} \\
    &= \sum_{r=1}^{0.1\log N} \sum_{i \in I_1} g^r_{\zeta_i}(x,y) \\
    &= \sum_{r=1}^{0.1\log N} \sum_{i \in I_1} \Delta^r_{v_i}(x,y)\int_{\zeta_i}
    \Delta^r_{v_i}\qquad\text{where }\zeta_i=\pm[v_i].
  \end{align*}

  Notice that $\int_{Z_1} g^r_{st}$ is always nonnegative: roughly speaking, it
  measures the square of the ``imbalance'' of positive and negative points in
  $Q^r_{st}$.  In particular, it's not hard to see that
  $\mathbb{E}(\int_{Z_1} g^r_{st})=c_{n,k} \cdot 2^{-2r}N$, and therefore
  $\mathbb{E}(\int_{Z_1} g)=c_{n,k}N\log N$.

  On the other hand, the derivative of $g$ is $O(\sqrt{N \log N})$ on average,
  but can be much larger in some places.  To remedy this, Ajtai, Koml\'os, and
  Tusn\'ady introduced a ``stopping time'' rule, building $f$ as the sum of some,
  but not all of the $g^r_{st}$.  We do not need to give the exact definition,
  remarking only that the function $f$ satisfies
  \begin{align}
    \mathbb{E}\bigl(\textstyle{\int_{Z_1}} f\bigr)
    &\geq \frac{c_{n,k}}{L}N\log N \label{2:int} \\
    \Lip f &\leq C_{n,k}\sqrt{N \log N}. \label{2:Lip}
  \end{align}

  Now, for $j \neq 1$,
  \[\bigl\lvert\mathbb{E}\bigl({\textstyle\int_{Z_j}} f\bigr)\bigr\rvert
  \leq \sum_{i \in I_j} \sum_{r=1}^{0.1\log N} \sum \bigl\{\bigl\lvert\mathbb{E}\bigl(
       {\textstyle\int_{\zeta_i}} g^r_{\zeta_{i'}}\bigr)\bigr\rvert
  \mid \zeta_{i'}\text{ is correlated with }\zeta_i\bigr\}.\]
  By Lemma \ref{correlation}, the value of each term of this triple sum is
  $O(2^{-r/2})$.  Therefore,
  \[\bigl\lvert\mathbb{E}\bigl({\textstyle\int_{Z_j}} f\bigr)\bigr\rvert
  \leq C_{n,k}LN \sum_{r=1}^{0.1\log N} 2^{-r/2} \leq (1+\sqrt{2})C_{n,k}LN.\]
  Combining this with \eqref{2:int} and \eqref{2:Lip}, we see that
  \[\mathbb{E}\Bigl(\int_{Z \cap P_{\vec x}} f\Bigr)
  \geq \frac{c_{n,k}}{L}\sqrt{N\log N}\Lip f\qquad\text{for large enough }N.\]

  \subsection*{Case $d=1$}
  We split the interval $[0,1]$ into $R$ equal regions, with $R$ to be determined
  later.  Write $\Delta_s=\Delta_{[s/R,(s+1)/R]}$, and let
  \[g(x)=\sum_{s=0}^{R-1} \Delta_s(x){\textstyle\int_{Z_1}}
  \chi_{\left[\frac{s}{R},\frac{s+1}{R}\right]}=\sum_{i \in I_1} g_{\zeta_i}(x)=
  \sum_{i \in I_1} \pm\Delta^{1/R}_{y_i}(x)\qquad\text{where }\zeta_i=\pm[y_i].\]
  We obtain the desired function $f$ by replacing
  $\int_{Z_1} \chi_{\left[\frac{s}{R},\frac{s+1}{R}\right]}$ with
  \[h_s=\sgn\biggl(\int_{Z_1} \chi_{\left[\frac{s}{R},\frac{s+1}{R}\right]}\biggr)
  \min\biggl\{\biggl\lvert\int_{Z_1} \chi_{\left[\frac{s}{R},\frac{s+1}{R}\right]}
  \biggr\rvert, C_{n,k}\sqrt{\frac{N}{R}}\biggr\}\]
  for some sufficiently large $C_{n,k}$.  Then $\Lip f \leq C_{n,k}\sqrt{NR}$ and
  $\int_{Z_1} f \geq \frac{c_{n,k}}{L}N$.

  On the other hand, for $j \neq 1$ and any $i \in I_j$ and $i' \in I_1$ such
  that $\zeta_i$ and $\zeta_{i'}$ are correlated, by Lemma \ref{correlation},
  $\bigl\lvert\mathbb{E}\bigl({\textstyle\int_{\zeta_i}} g_{\zeta_{i'}}\bigr)
  \bigr\rvert \leq C_{n,k}R^{-1/2}$, and therefore
  \[\bigl\lvert\mathbb{E}\bigl({\textstyle\int_{Z_j}} f\bigr)\bigr\rvert
  \leq C_{n,k}LNR^{-1/2}.\]
  For some large enough $R$, depending on $n$ and $k$ but not on $N$,
  \[\bigl\lvert\mathbb{E}\bigl({\textstyle\int_{Z_j}} f\bigr)\bigr\rvert
  \leq \frac{1}{2L^2}\mathbb{E}\bigl({\textstyle\int_{Z_1}} f\bigr).\]
  Thus $\mathbb{E}\bigl(\int_{Z \cap P_{\vec x}} f\bigr) \geq C_{n,k}\sqrt{N}\Lip f$,
  completing the proof.
\end{proof}

\section{Concentration of measure} \label{S:concentration}

In this section, we show that when $n-k \geq 2$, the size of the filling tends to
concentrate around its mean.  That is, we show that \eqref{AKTstats:V} holds in
the case of Theorems \ref{sphere}, \ref{cube}, and \ref{knot}.  We first prove
this in the case of Theorem \ref{AKT:boundary}.  The main tool is the
concentration of measure in high-dimensional balls, an idea due to Gromov and
Milman \cite{GroMi} and of wide importance in probability theory \cite{Ledoux}.
We follow the exposition due to Bobkov and Ledoux \cite[\S7.1]{BobLMem} which
covers the 1-dimensional case; the higher-dimensional cases are essentially the
same although they do not seem to appear explicitly in the literature.
\begin{thm} \label{concentration}
  Let $Z$ be a random cycle in $C_0([0,1]^n,\partial[0,1]^n)$ as in Theorem
  \ref{AKT:boundary}.  Then for every $r>0$,
  \[\mathbb{P}[\lvert FV(Z)-\mathbb{E}(FV(Z))\rvert \geq r] \leq
  C_1\exp\bigl(-C_2r/\sqrt{N}\bigr)\]
  for universal constants $C_1,C_2>0$.
\end{thm}
In particular, the standard deviation of $FV(Z)$ is at most $O(\sqrt{N})$.  In
other words, for $n \geq 2$, $FV(Z)/\mathbb{E}(FV(Z))$ converges to 1 as
$N \to \infty$.
\begin{proof}
  Equip $X=C_0([0,1]^n,\partial[0,1]^n)$ with the metric
  \[d_{FV}(Z,Z')=FV(Z-Z')\]
  and let $E=[-1,1] \times [0,1]^{n-1}$.  Define $\zeta_0:E \to X$ by
  \[\zeta_0(\pm x_1,x_2,\ldots,x_n)=\pm[(x_1,x_2,\ldots,x_n)]\]
  and $\zeta:(E^N,d_{\text{Eucl}}) \to X$ by
  \[\zeta(v_1,\ldots,v_N)=\sum_{i=1}^N \zeta_0(v_i).\]
  This map is $\sqrt{N}$-Lipschitz since when every point moves by a tiny amount
  $\epsi$, the distance is $\sqrt{N}\epsi$ in the domain and $N\epsi$ in the
  range.

  Define the \emph{concentration function} of a metric measure space $(M,d,\mu)$
  of total measure 1 to be
  \[\conc_M(r)=\sup\{1-\mu(N_r(A)) \mid \mu(A) \geq 1/2\}, \qquad r>0,\]
  where $N_r(A)$ is the $r$-neighborhood of the set $A$.  The key observation of
  Gromov and Milman \cite[Thm.~4.1]{GroMi} is that
  \[\conc_M(r) \leq \frac{3}{4}e^{-\ln(3/2)\lambda_1 r},\]
  where $\lambda_1$ is the first nonzero eigenvalue of the Laplacian on $M$.
  Since the spectrum of a product of manifolds is the sum of its spectra,
  $\lambda_1$ is constant on powers of $M$.  The map $\zeta$
  is measure-preserving, so it follows that
  \[\conc_X(r) \leq \frac{3}{4}\exp\bigl(-\ln(3/2)\lambda_1 r/\sqrt{N}\bigr).\]
  Therefore, for any $1$-Lipschitz function $u:X \to \mathbb{R}$,
  \[\mathbb{P}[|u(Z)-\operatorname{median}(u)| \geq r] \leq
  \frac{3}{2}\exp\bigl(-\ln(3/2)\lambda_1 r/\sqrt{N}\bigr).\]
  By Chebyshev's inequality, the same, modulo constants, holds for the mean (see
  also \cite[Prop.~1.10]{Ledoux}).
\end{proof}
To adapt this proof for the case of Theorems \ref{sphere}, \ref{cube}, and
\ref{knot}, we just have to change the space $E$: take
\begin{align*}
  E_{\text{sphere}} &= \widetilde{Gr}_k(\mathbb{R}^n) &&
  \text{for Theorem \ref{sphere}} \\
  E_{\text{cube}} &= \{\text{affine $k$-planes }P \subset \mathbb{R}^n
  \mid P \cap [0,1]^n \neq \emptyset\} &&
  \text{for Theorem \ref{cube}} \\
  E_{\text{knot}} &= [0,1]^n &&
  \text{for Theorem \ref{knot}}.
\end{align*}
In the first two cases, the map $\zeta$ is constructed as before.  In the last
case, for a $k$-pseudomanifold $M$ with vertex set $M^0$,
$\zeta_M:E_{\text{knot}}^{\lvert M^0 \rvert} \to Z_k([0,1]^n)$ sends
$(v_1,\ldots,v_{|M^0|})$ to the image of the linear immersion of $M$ with vertices
$(v_1,\ldots,v_{|M^0|})$.

In each case, it is easy to see that if the space of $k$-cycles is given the
filling volume metric, then $\zeta$ is $\sqrt{N}$-Lipschitz.  Therefore, the rest
of the proof is identical to that of Theorem \ref{concentration}.

\bibliographystyle{amsalpha}
\bibliography{stochastic}

\newcommand{\etalchar}[1]{$^{#1}$}
\providecommand{\bysame}{\leavevmode\hbox to3em{\hrulefill}\thinspace}
\providecommand{\MR}{\relax\ifhmode\unskip\space\fi MR }
\providecommand{\MRhref}[2]{%
  \href{http://www.ams.org/mathscinet-getitem?mr=#1}{#2}
}
\providecommand{\href}[2]{#2}
\begin{thebibliography}{Tanaka}

\bibitem[ABD{\etalchar{+}}]{ABDKS}
J.~Arsuaga, T.~Blackstone, Y.~Diao, E.~Karadayi, and M.~Saito, \emph{Linking of
  uniform random polygons in confined spaces}, J. Phys. A \textbf{40} (2007),
  no.~9, 1925--1936.

\bibitem[AKT]{AKT}
M.~Ajtai, J.~Koml{\'o}s, and G.~Tusn{\'a}dy, \emph{On optimal matchings},
  Combinatorica \textbf{4} (1984), no.~4, 259--264.

\bibitem[Bany]{Banyaga}
A.~Banyaga, \emph{Formes-volume sur les vari\'{e}t\'{e}s \`a bord}, Enseign.
  Math. (2) \textbf{20} (1974), 127--131.

\bibitem[BL1]{BobLMem}
S.~Bobkov and M.~Ledoux, \emph{One-dimensional empirical measures, order
  statistics, and {K}antorovich transport distances}, Mem. Amer. Math. Soc.
  \textbf{261} (2019), no.~1259, v+126.

\bibitem[BL2]{BobL}
\bysame, \emph{A simple {F}ourier analytic proof of the {AKT} optimal matching
  theorem}, arXiv:1909.06193 [math.PR], 2019.

\bibitem[BMPR]{BMPR}
M.~Bruveris, P.~W. Michor, A.~Parusi\'{n}ski, and A.~Rainer, \emph{Moser's
  theorem on manifolds with corners}, Proc. Amer. Math. Soc. \textbf{146}
  (2018), no.~11, 4889--4897.

\bibitem[BottTu]{BottTu}
R.~Bott and L.~W. Tu, \emph{Differential forms in algebraic topology}, Graduate
  Texts in Mathematics, no.~82, Springer, 1982.

\bibitem[EPC{\etalchar{+}}]{ECHLP}
D.~Epstein, M.~Paterson, J.~Cannon, D.~Holt, S.~Levy, and W.~P. Thurston,
  \emph{Word processing in groups}, Jones and Bartlett, 1992.

\bibitem[E-Z]{EZ}
Ch. Even-Zohar, \emph{Models of random knots}, J. Appl. Comput. Topol.
  \textbf{1} (2017), no.~2, 263--296.

\bibitem[Fed]{FedBk}
H.~Federer, \emph{Geometric measure theory}, Die Grundlehren der mathematischen
  Wissenschaften, Band 153, Springer-Verlag New York Inc., New York, 1969.

\bibitem[FF]{FF}
H.~Federer and W.~H. Fleming, \emph{Normal and integral currents}, Ann. of
  Math. (2) \textbf{72} (1960), 458--520.

\bibitem[FK]{FlKo}
E.~Flapan and K.~Kozai, \emph{Linking number and writhe in random linear
  embeddings of graphs}, J. Math. Chem. \textbf{54} (2016), no.~5, 1117--1133.

\bibitem[FKW]{FKW}
P.~Franek, M.~Kr\v{c}\'{a}l, and H.~Wagner, \emph{Solving equations and
  optimization problems with uncertainty}, J. Appl. Comput. Topol. \textbf{1}
  (2018), no.~3-4, 297--330.

\bibitem[GM]{GroMi}
M.~Gromov and V.~D. Milman, \emph{A topological application of the
  isoperimetric inequality}, American Journal of Mathematics \textbf{105}
  (1983), no.~4, 843--854.

\bibitem[Gro]{GroHED}
M.~Gromov, \emph{Homotopical effects of dilatation}, J. Differential Geom.
  \textbf{13} (1978), no.~3, 303--310.

\bibitem[HPZ]{HPZ}
N.~Holden, Y.~Peres, and A.~Zhai, \emph{Gravitational allocation on the
  sphere}, Proc. Natl. Acad. Sci. USA \textbf{115} (2018), no.~39, 9666--9671.

\bibitem[HS]{HS}
R.~Hardt and L.~Simon, \emph{Boundary regularity and embedded solutions for the
  oriented {P}lateau problem}, Ann. of Math. (2) \textbf{110} (1979), no.~3,
  439--486.

\bibitem[Jones]{Jones}
P.~W. Jones, \emph{Rectifiable sets and the traveling salesman problem},
  Invent. Math. \textbf{102} (1990), no.~1, 1--15.

\bibitem[Led]{Ledoux}
M.~Ledoux, \emph{The concentration of measure phenomenon}, Mathematical Surveys
  and Monographs, vol.~89, American Mathematical Society, Providence, RI, 2001.

\bibitem[Marko]{Marko}
J.~F. Marko, \emph{Linking topology of tethered polymer rings with applications
  to chromosome segregation and estimation of the knotting length}, Phys. Rev.
  E \textbf{79} (2009), 051905.

\bibitem[Mil]{Millett}
K.~C. Millett, \emph{Monte {C}arlo explorations of polygonal knot spaces},
  Knots in {H}ellas '98 ({D}elphi), Ser. Knots Everything, vol.~24, World Sci.
  Publ., River Edge, NJ, 2000, pp.~306--334.

\bibitem[Mor]{Morgan}
F.~Morgan, \emph{Geometric measure theory: {A} beginner's guide}, fourth ed.,
  Elsevier/Academic Press, Amsterdam, 2009.

\bibitem[Moser]{Moser}
J.~Moser, \emph{On the volume elements on a manifold}, Trans. Amer. Math. Soc.
  \textbf{120} (1965), 286--294.

\bibitem[Shor]{Shor}
P.~W. Shor, \emph{The average-case analysis of some on-line algorithms for bin
  packing}, Combinatorica \textbf{6} (1986), no.~2, 179--200.

\bibitem[Tal1]{TalHiD}
M.~Talagrand, \emph{Matching random samples in many dimensions}, Ann. Appl.
  Probab. \textbf{2} (1992), no.~4, 846--856.

\bibitem[Tal2]{TalHiD2}
\bysame, \emph{The transportation cost from the uniform measure to the
  empirical measure in dimension {$\ge 3$}}, Ann. Probab. \textbf{22} (1994),
  no.~2, 919--959.

\bibitem[Tal3]{TalBk}
\bysame, \emph{Upper and lower bounds for stochastic processes: Modern methods
  and classical problems}, Ergebnisse der Mathematik und ihrer Grenzgebiete. 3.
  Folge. A Series of Modern Surveys in Mathematics, vol.~60, Springer,
  Heidelberg, 2014.

\bibitem[Tanaka]{Tanaka}
F.~Tanaka, \emph{{Gauge Theory of Topological Entanglements. I: --- General
  Theory ---}}, Progress of Theoretical Physics \textbf{68} (1982), no.~1,
  148--163.

\bibitem[Young]{Young}
R.~Young, \emph{Quantitative nonorientability of embedded cycles}, Duke
  Math. J. \textbf{167} (2018), no.~1, 41--108.

\end{thebibliography}
\end{document}